\documentclass[reqno, 11pt]{amsart}
\usepackage[top=4cm, bottom=3cm, left=3.2cm, right=3.2cm]{geometry}
\usepackage{enumerate}
\usepackage{amssymb,amsmath}
\usepackage{slashed}
\usepackage{color}
\newtheorem{thm}{Theorem}

\newtheorem{lem}{Lemma}

\newtheorem{defn}{Definition}[section]
\newtheorem{rem}{Remark}
\allowdisplaybreaks

\numberwithin{equation}{section} \numberwithin{lem}{section}
\numberwithin{thm}{section} \numberwithin{prop}{section}
\numberwithin{cor}{section} \numberwithin{rem}{section}
\title[Bipolar hydrodynamic model]{Subsonic steady-states for bipolar hydrodynamic model for semiconductors}
\author{Siying Li$^{\rm 1}$, Ming Mei$^{\rm 2, 3}$, Kaijun Zhang$^{\rm 1}$ and Guojing Zhang$^{\rm 1, *}$}
\thanks{*Corresponding author: Guojing Zhang.}
\thanks{E-mail addresses: lisiying630@163.com (S. Li), ming.mei@mcgill.ca (M. Mei), zhangkj201@nenu.edu.cn (K. Zhang), zhanggj100@nenu.edu.cn (G. Zhang).}

\begin{document}
  \maketitle
  \begin{center}
	{\footnotesize
		$^{\rm 1}$ School of Mathematics and Statistics, Northeast Normal University,\\
		Changchun, 130024, P.R.China.\\
		\smallskip
		$^{\rm 2}$ Department of Mathematics, Champlain College Saint-Lambert,\\
		Saint-Lambert, Quebec, J4P 3P2, Canada.\\
		$^{\rm 3}$ Department of Mathematics and Statistics, McGill University,\\
		Montreal, Quebec, H3A 2K6, Canada.}
  \end{center}
  \maketitle
  \date{}
	
\begin{abstract}
In this paper, we study the well-posedness, ill-posedness and uniqueness of the stationary 3-D radial solution to the bipolar isothermal hydrodynamic model for semiconductors. The density of electron is imposed with sonic boundary and interiorly subsonic case and the density of hole is fully subsonic case. It is difficult to estimate the upper and lower bounds of the holes due to the coupling of electrons and holes and the degeneracy of electrons at the boundary. Thus, we use the topological degree method to prove the well-posedness of solution. We prove the ill-posedness of subsonic solution under some conditions by direct mathematical analysis and contradiction method. The ill-posedness property shows significant difference to the unipolar model. Another highlight of this paper is the application of specific energy method to obtain the uniqueness of solution in two cases. One case is the relaxation time $\tau=\infty$, namely, the pure Euler-Poisson case; the other case is $\frac{j}{\tau}\ll 1$, which means that, when the current flow is sufficiently small and the relaxation time is sufficiently large both can satisfy.  
\end{abstract}


\section{Introduction}

The bipolar hydrodynamic (HD) model, proposed first by Bløtekjær in \cite{B1973}, is usually described for the charged fluid particles such as electrons and holes in semiconductor devices \cite{B1973, J2001, MRS1989}, which is written as the following system of Euler-Poisson equations:
\begin{equation}
	\left\{
	\begin{aligned}\label{eq1}
		&\rho _t+{\rm div} (\rho \vec{u})=0,\\
		&(\rho \vec{u})_t+{\rm div} (\rho \vec{u}\otimes\vec{u})+\nabla P_1(\rho)=\rho \vec{E}-\frac{\rho \vec{u}}{\tau},\\
		&n _t+{\rm div} (n \vec{v})=0,\\
		&(n\vec{v})_t+{\rm div} (n\vec{v}\otimes\vec{v})+\nabla P_2(n)=-n \vec{E}-\frac{n\vec{v}}{\tau},\\
		&\nabla\cdot \vec{E}=\rho -n-b(x).
	\end{aligned}
	\right.
\end{equation}
Here, $\rho$, $n$, $\vec{u}$, $\vec{v}$ and $\vec{E}$ represent the electron density, the hole density, the electron velocity, the hole velocity and the electric field, respectively. $P_1(\rho)$ and $P_2(n)$ are given functions which denote the pressure of electron and the pressure of hole. When the system is isothermal, the physical representation of the pressure functions are
\begin{align}\label{eq4} 
	P_1(\rho)=T\rho,\,\,\,\, P_2(n)=Tn,
\end{align}
where $T>0$ is the constant temperature. The function $b(x)>0$ is the doping profile denoting the density of impurities in semiconductor devices. The constant $\tau>0$ stands for the relaxation time.

The corresponding steady-state equation of \eqref{eq1} is as follows
\begin{equation}
	\left\{
	\begin{aligned}\label{eq3}
		&{\rm div} (\rho \vec{u})=0,\\
		&{\rm div} (\rho \vec{u}\otimes\vec{u})+\nabla P_1(\rho)=\rho \vec{E}-\frac{\rho \vec{u}}{\tau},\\
		&{\rm div} (n \vec{v})=0,\\
		&{\rm div} (n\vec{v}\otimes\vec{v})+\nabla P_2(n)=-n \vec{E}-\frac{n\vec{v}}{\tau},\\
		&\nabla\cdot \vec{E}=\rho -n-b(x).
	\end{aligned}
	\right.
\end{equation}

There are many researches on the stationary solution for the HD model of semiconductors. In 1990, Degond and Markowich \cite{DM1990} first proved the existence of subsonic solution for one-dimensional case, and obtained the uniqueness of solution with small electric current. Subsequently, lots of attentions has been paid to the steady subsonic flows with different boundary conditions as well as the higher dimensional case in \cite{BDC2014, BDC2016, DM1993, GS2005, H2011, HMWY2011, MWZ2021, NS2007, NS2009}. For the supersonic flows, Peng and Violet \cite{PV2006} showed the existence and uniqueness of supersonic solution with a strong supersonic background for one-dimensional model. Bae et al \cite{Bae} extended this work to two-dimensional case for pure Euler-Poisson equation, namely, the semiconductor effect is zero. The transonic flows has also been extensively studied in \cite{AMPS1991, DZ2020, G1992, GM1996, LRXX2011, LX2012, R2005}. Li-Mei-Zhang-Zhang \cite{LMZZ2017, LMZZ2018} proved in great depth the structure of all types with the sonic boundary when the doping profile is subsonic and supersonic, respectively. The sonic boundary condition means the system has degeneracy effect, which makes the system has strong singularity. Chen-Mei-Zhang-Zhang \cite{CMZZ2020} extended the
corresponding results to the transonic doping profile. Also, Chen et al \cite{CMZZ2021, CMZZ2022, CMZZ2023} studied the radial or the spiral radial subsonic, supersonic and transonic solutions in two and three dimensional spaces. Further, the existence and the regularity of smooth transonic solutions are also investigated by Wei et al in \cite{WMZZ2021}. Recently, Feng et al showed the structural stability of different types of solutions in \cite{FHM2022, FMZ2023}, respectively. Asymptotic limits of subsonic or sonic-subsonic solutions were studied in \cite{CLMZ2023, CG2000, P2002, P2003, PW2004}. 

However, the corresponding results for bipolar model are very limited. For isentropic case, Zhou and Li \cite{ZL2009} proved the existence and uniqueness of stationary solution with Dirichlet boundary conditions in one-dimensional space when the doping profile is zero. For isothermal case, Tsuge \cite{T2010} obtained the existence and uniqueness of the subsonic stationary solution in one-dimensional space for the electrostatic potential is small enough. Yu \cite{Y2017} studied the existence and uniqueness of the sbusonic stationary solution with insulating boundary conditions by the calculus of variations. Recently, Mu-Mei-Zhang \cite{MMZ2020} used the topological degree method to prove the well-posedness and ill-posedness of stationary subsonic and supersonic solutions with the electrons sonic boundary. The existence obtained in \cite{MMZ2020} relies strongly on the hole far from the sonic line, and the ill-posedness is partially given only for the case of $\tau=\infty$ and the density of hole is close enough to the sonic value. Moreover, the uniqueness of this kind of solution is still unclear and it will be very difficult to prove. This shows the bipolar model is significantly different to the unipolar model. As known to all, the bipolar model has more importance in physical practice. In this paper, we are devoted to studying the well-posedness, ill-posedness and uniqueness of radial solution to the steady-state equation \eqref{eq3} in the domain of 3-D hollow ball. 

Let us denote
\begin{align*}
	&r=|x|=\sqrt{{x_1}^2+{x_2}^2+{x_3}^2},\\
	&\rho=\rho(|x|)=\rho(r),\,\,\,\,\,\, n=n(|x|)=n(r),\\
	&\vec{u}=u(r)\cdot \frac{x}{r},\,\,\,\,\,\, \vec{v}=v(r)\cdot \frac{x}{r},\\
	&\vec{E}=E(r)\cdot\frac{x}{r},\\
	&b(x)=b(r),\\
	&\vec{J}:=\rho\vec{u}=\rho(r)u(r)\cdot\frac{x}{r}=J(r)\cdot\frac{x}{r},\,\,\,\,\,\, \mbox{the current density of electrons},\\ &\vec{K}:=n\vec{v}=n(r)v(r)\cdot\frac{x}{r}=K(r)\cdot\frac{x}{r},\,\,\,\,\,\, \mbox{the current density of holes}.   
\end{align*}
Then the system \eqref{eq3} becomes
\begin{equation}
	\left\{
	\begin{aligned}\label{eq2}
		&J_r+\frac{2J}{r} =0, \mbox{namely}, J=\frac{j_1}{r^2},\\
		&\frac{1}{2}\rho\left(\frac{J^2}{\rho^2}\right)_r+P_1(\rho)_r=\rho E-\frac{J}{\tau},\\
		&K_r+\frac{2K}{r} =0, \mbox{namely}, K=\frac{j_2}{r^2}\\
		&\frac{1}{2}n\left(\frac{K^2}{n^2}\right)_r+P_2(n)_r=-n E-\frac{K}{\tau},\\
		&E_{r}+\frac{2}{r}E=\rho-n-b(r),
	\end{aligned}
	\right.
\end{equation}
where $j_1$ and $j_2$ denote two constants. We can see that $r=0$ is the singular point, so we consider the system in hollow ball, namely, $r \in [\varepsilon_0,1]$ for any given $\varepsilon_0>0$ in this paper.

According to the terminology from gas dynamics, we call $c_e:=\sqrt{P_1^{'}(\rho)}=\sqrt{T}>0$ the sound speed of electron and and $c_h:=\sqrt{P_2^{'}(n)}=\sqrt{T}>0$ the sound speed of hole by \eqref{eq4}. Thus, the corresponding electron velocity $u$ and hole velocity $v$ of the system \eqref{eq2} are said to be subsonic (or sonic) if 
\begin{align}\label{eq5}
u=\frac{|J|}{\rho}<(=)c_e=\sqrt{P_1^{'}(\rho)}=\sqrt{T}\,\, \mbox{and}\,\, v=\frac{|K|}{n}<(=)c_h=\sqrt{P_2^{'}(n)}=\sqrt{T}.
\end{align}

Without loss of generality, we assume that $j_1=j$, $j_2=-j$ in \eqref{eq2}, and take $j=1$, $T=1$. Now dividing $\eqref{eq2}_2$ and $\eqref{eq2}_4$ by $\rho$ and $n$, respectively, we obtain 
\begin{equation}
	\left\{
	\begin{aligned}\label{eq7}
		&\frac{1}{2}\rho\left(\frac{1}{r^4\rho^2}\right)_r+\rho_r=\rho E-\frac{1}{\tau r^2},\\
		&\frac{1}{2}n\left(\frac{1}{r^4 n^2}\right)_r+n_r=-nE+\frac{1}{\tau r^2},\\
		&(r^2E)_r=r^2(\rho-n-b(r)).
	\end{aligned}
	\right.
\end{equation}

For the sake of simplicity, we introduce two new variables
\begin{align}\label{eq8}
	g(r):=r^2\rho(r),\,\,\,\, m(r):=r^2n(r)	
\end{align}
and define 
\begin{align}\label{eq9}
	B(r):=r^2b(r).	
\end{align}

Substituting \eqref{eq8}-\eqref{eq9} into \eqref{eq7}, we have
\begin{equation}
	\left\{
	\begin{aligned}\label{eq10}
		&\left(\frac{1}{g}-\frac{1}{g^3}\right)g_r+\frac{1}{\tau g}-\frac{2}{r}=E,\\
		&\left(\frac{1}{m}-\frac{1}{m^3}\right)m_r-\frac{1}{\tau m}-\frac{2}{r}=-E,\\
		&\left(r^2E\right)_r=g-m-B(r).
	\end{aligned}
	\right.
\end{equation}

Now from \eqref{eq5}, it is clear by a series of simple calculations that the stationary flow of electron and hole to \eqref{eq10} are called to be subsonic (or sonic) if
\begin{align}\label{eq11} 
	u<(=)1,\,\,\, \mbox{or equivalently},\,\,\, g>(=)1	
\end{align}
and
\begin{align}\label{eq12}
	v<(=)1,\,\,\, \mbox{or equivalently},\,\,\, m>(=)1.	
\end{align}

By \eqref{eq11}-\eqref{eq12}, we impose the sonic boundary conditions to \eqref{eq10} with electron
\begin{align}\label{eq13}
	g(\varepsilon_0)=g(1)=1,	
\end{align}
and a given boundary condition to hole is proposed as
\begin{align}\label{eq14}
	m(\varepsilon_0)=\eta_0,	
\end{align}
where the value of $\eta_0$ will be specified later. 

Next, we need to study the well-posedness, ill-posedness and uniqueness of solution to \eqref{eq10} with boundary conditions \eqref{eq13}-\eqref{eq14}, and the density of electron is considered interiorly subsonic case $(namely, \, g>1 \, \mbox{on} \, (\varepsilon_0,1))$ and the density of hole is considered fully subsonic case $(namely, \, m>1 \, \mbox{on}\, [\varepsilon_0,1])$.

Multiplying $\eqref{eq10}_1$ and $\eqref{eq10}_2$ by $r^2$ and taking the derivative with respect to $r$, then according to $\eqref{eq10}_3$, we get
\begin{equation}
	\left\{
	\begin{aligned}\label{eq15}
		&\left[r^2\left(\frac{1}{g}-\frac{1}{g^3}\right)g_r+\frac{r^2}{\tau g}-2r\right]_r=g-m-B(r),\\
		&\left[r^2\left(\frac{1}{m}-\frac{1}{m^3}\right)m_r-\frac{r^2}{\tau m}-2r\right]_r=m+B(r)-g,\\
		&g(\varepsilon_0)=g(1)=1,m(\varepsilon_0)=\eta_0.
	\end{aligned}
	\right.
\end{equation}
Adding equations $\eqref{eq15}_1$ to $\eqref{eq15}_2$, and dividing both sides by $r^2$ for the resultant equation, we derive 
\begin{align*}
\left(\frac{1}{g}-\frac{1}{g^3}\right)g_r+\left (\frac{1}{m}-\frac{1}{m^3}\right)m_r+\frac{1}{\tau}\left(\frac{1}{g}-\frac{1}{m}\right)-\frac{4}{r}=\frac{c}{r^2}.
\end{align*}
Integrating the above equation on $[\varepsilon _0,1]$, we obtain
\begin{align}\label{c1}
w(m(1))-w(\eta_0)+\frac{1}{\tau}\int_{\varepsilon_0}^{1}\left(\frac{1}{g}-\frac{1}{m}\right)dr+4{\rm ln}\varepsilon_0=c\cdot\frac{1-\varepsilon_0}{\varepsilon_0},
\end{align}
where $w(h):={\rm ln}h+\frac{1}{2h^2}$.

By $\eqref{eq10}_1$ and $\eqref{eq10}_2$, we can get
\begin{align}\label{c2}
\left(\frac{1}{g}-\frac{1}{g^3}\right)g_r+\left (\frac{1}{m}-\frac{1}{m^3}\right)m_r+\frac{1}{\tau}\left(\frac{1}{g}-\frac{1}{m}\right)-\frac{4}{r}=0.	
\end{align}

The equation \eqref{eq10} is equivalent to the equation \eqref{eq15} if and only if $c=0$. Now let us define
\begin{align}\label{c3}
w(\eta_1)-w(\eta_0)+\frac{1}{\tau}\int_{\varepsilon_0}^{1}\left(\frac{1}{g}-\frac{1}{m}\right)dr+4{\rm ln}\varepsilon_0=0.
\end{align}
Therefore, the equations \eqref{eq10} and \eqref{eq15} are equivalent if and only if 
\begin{align}\label{y9}
m(1)=\eta_1.
\end{align}

Notice that the model \eqref{eq15} is denegerate at the sonic boudary $g(\varepsilon_0)=g(1)=1$, the solution of \eqref{eq15} will lose certain regularity, and have to be in the weak form. Thus, we give the following definition of weak solution by \cite{LMZZ2017}.
\begin{defn}\label{def1}
	$(g(r), m(r))$ is called a pair of subsonic solutions, which means that interiorly subsonic solution $g(r)$ coupled with fully subsonic solution $m(r)$ of system \eqref{eq15} with $g(\varepsilon_0)=g(1)=1,m(\varepsilon_0)=\eta_0$, $\left(g(r)-1\right)^2\in H_0^1(\varepsilon_0,1),m(r)\in W^{2,\infty}(\varepsilon_0,1)$, and for any $\varphi\in H_0^1(\varepsilon_0,1)$ it holds that 
	\begin{align}\label{eq16}
		\int_{\varepsilon_0}^{1}\left[r^2\left(\frac{1}{g}-\frac{1}{g^3}\right)g_r+\frac{r^2}{\tau g}-2r\right]\cdot \varphi_rdr+\int_{\varepsilon_0}^{1}\left(g-m-B(r)\right)\varphi dr=0	
	\end{align}  
	and 
	\begin{align}\label{eq17}
		\int_{\varepsilon_0}^{1}\left[r^2\left(\frac{1}{m}-\frac{1}{m^3}\right)m_r-\frac{r^2}{\tau m}-2r\right]\cdot \varphi_rdr+\int_{\varepsilon_0}^{1}\left(m+B(r)-g\right)\varphi dr=0.	
	\end{align} 
\end{defn}

\begin{rem}\label{r1}
	The identity \eqref{eq16} is well-defined for $\left(g(r)-1\right)^2\in H_0^1(\varepsilon_0,1)$ and $\varphi\in H_0^1(\varepsilon_0,1)$, which is equivalent to 
	\begin{align*}
		\int_{\varepsilon_0}^{1}\left[r^2\cdot\frac{g+1}{2g^3}((g-1)^2)_r+\frac{r^2}{\tau g}-2r\right]\cdot \varphi_rdr+\int_{\varepsilon_0}^{1}\left(g-m-B(r)\right)\varphi dr=0.	
	\end{align*}
	
	Once $(g(r),m(r))$ is obtained in equation \eqref{eq15}, then $(\rho(r),n(r))$ can be determined by \eqref{eq8}. According to \eqref{eq10}, we can further get the solution of the electric field $E(r)$ 
	\begin{align*}
	E(r)=\left(\frac{1}{g}-\frac{1}{g^3}\right)g_r+\frac{1}{\tau g}-\frac{2}{r}=\frac{(g+1)[(g-1)^2]_r}{2g^3}+\frac{1}{\tau g}-\frac{2}{r}.	
    \end{align*}

    Therefore, solving \eqref{eq7} amounts to finding the solution to \eqref{eq15} satisfying \eqref{c3}-\eqref{y9}.	
\end{rem}
Throughout the paper we assume that the doping profile $B(r)\in L^\infty(\varepsilon_0,1)$ and denote
\begin{align*}
	\underline{B}:=\underset{r\in(\varepsilon_0,1)}{{\rm essinf}}\,\, B(r)\,\,\,\, \mbox{and}\,\,\,\, \bar{B}:=\underset{r\in(\varepsilon_0,1)}{{\rm esssup}}\,\, B(r).
\end{align*}
Our main results are stated below.
\begin{thm}\label{th1}
For any $B(r)\in L^\infty(\varepsilon_0,1)$, $\tau>0$ and $\underline{m}>1$, there exists a constant $\eta^\ast(\underline{m}, \bar{B}, \tau)>1$ which only depends on $\underline{m}, \bar{B}$ and $\tau$, such that for any $\eta_0\ge \eta^\ast$, there are the pair of subsonic solutions $(g,m)\in C^{\frac{1}{2}}[\varepsilon_0,1]\times W^{2,\infty}(\varepsilon_0,1)$ and $m\ge \underline{m}+3$ on $[\varepsilon_0,1]$ to \eqref{eq15} with the conditions of \eqref{c3}-\eqref{y9}.
\end{thm}

\begin{thm}\label{th2}
For any $\bar{\eta}>1$, there exists $B^*=B^*(\bar{\eta})>1$ such that if $\eta_0<\bar{\eta}$ and $\underset{r\in(\alpha, \beta)}{{\rm inf}}B\ge B^*$, there is no subsonic solution to \eqref{eq10}, \eqref{eq13}-\eqref{eq14}. Here $\alpha=\varepsilon_0+\frac{1-\varepsilon_0}{4}$, $\beta=\varepsilon_0+\frac{3(1-\varepsilon_0)}{4}$.
\end{thm}

\begin{thm}\label{th3}
The solution obtained in Theorem \ref{th1} for \eqref{eq7} and thus for \eqref{eq2} will be unique when $\tau=\infty$ and $\frac{j}{\tau}\ll 1$.  	
\end{thm}

\begin{rem}\label{r2}
1.~In this paper, we focus on the well-posedness, ill-posedness and uniqueness of the radial solution to \eqref{eq3} in 3-D hollow ball. In fact, the results are also applicable to 2-D case.

2.~We prove the ill-posedness of solution to \eqref{eq10} by direct mathematical analysis and contradiction method. It is totally different from the result of ill-posedness in \cite[Theorem 2.2]{MMZ2020}. The relaxation time $\tau$ is arbitrarily given constant in this paper. Our result show in great depth that the ill-posedness does not require $\eta_0$ close to $1$ or $\tau=+\infty$.

3.~We prove the uniqueness of solution in two cases. For the case of $\tau=\infty$, namely, the semiconductor effect is zero, we use the simple energy method to prove the uniqueness of solution. For the second case, the main difficulty lies on that the electron is degenerate at the boundary and the non-local terms caused by coupling of electrons and holes. To do this, we apply the method of exponential variation in \cite[Theorem 10.7]{GT2001} and make specific modifications based on the degeneracy of electrons at the boundary.
\end{rem}

This paper is arranged as follows. In Section 2, we show the well-posedness of the solution utilizing the topological degree method when the model \eqref{eq15} satisfying the boundary conditions \eqref{c3}-\eqref{y9}. In Section 3, we apply direct mathemaical analysis and contradiction method to prove the ill-posedness of solution when the doping profile $B(r)$ is sufficiently large and the holes boundary $\eta_0<\bar{\eta}$ for the system \eqref{eq10} with boundary conditions \eqref{eq13}-\eqref{eq14}. In Section 4, we are devoted to prove the uniqueness of solution for the two cases of $\tau=\infty$ and $\frac{j}{\tau}\ll 1$.

%%%%%%%%%%%%%%%%%%%%%%%%%%%%%%%%%%%%%%%%%%%%%%%%%%%%%%%%%%%%%%%%%%%%%%%%%%%%%%%%%%%%%%%%%%%%%%%%%%

\section{The well-posedness of subsonic solution}

Estimating the upper and lower bounds of the holes are difficult due to the couping of the system for electrons and holes. We apply the topological degree method to prove the well-posedness of solution to \eqref{eq15}, which is used in \cite{MMZ2020}. To do this, we first consider the following approximate system
\begin{equation}
	\left\{
	\begin{aligned}\label{yy41}
		&\left[r^2 \left( \frac{1}{g_k}-\frac{k}{g_k^3} \right)(g_k)_r \right]_r+ \left(\frac{r^2}{\tau g_k} \right)_r =g_k-m_k-B(r)+2,\\
		&\left[r^2 \left( \frac{1}{m_k}-\frac{1}{m_k^3} \right)(m_k)_r \right]_r- \left(\frac{r^2}{\tau m_k} \right)_r =m_k+B(r)-g_k+2,\\
		&g_k(\varepsilon _0)=g_k(1)=1,m_k(\varepsilon_0)=\eta _0,m_k(1)=\eta_1,
	\end{aligned}
	\right.
\end{equation}
where $r\in [\varepsilon_0,1]$ and $k\in (0,1)$ is a constant. Here, $\eta_1$ satisfies \eqref{c3} with $g=g_k$ and $m=m_k$.
\begin{lem}\label{le1}
For any $B(r)\in L^\infty(\varepsilon_0,1)$, $\tau>0$, $k\in(0,1)$ and $\underline{m}>1$, there exists a constant $\eta^\ast(\underline{m}, \bar{B}, \tau)>1$ which only depends on $\underline{m}, \bar{B}$ and $\tau$, such that for any $\eta_0\ge \eta^\ast$, the system \eqref{yy41} admits a pair of subsonic solutions $(g,m)\in W^{2,\infty}(\varepsilon_0,1)\times W^{2,\infty}(\varepsilon_0,1)$ and $m\ge \underline{m}+3$ on $[\varepsilon_0,1]$.	
\end{lem}
\begin{proof}
Define
\begin{align*}
	&U=\left \{(g,m)\in C[\varepsilon _0,1]\times C[\varepsilon _0,1]\right \},\\
	&V=\left \{(g,m)\in X,\lambda <g<\Lambda,l<m<L\right \},	
\end{align*}
where
\begin{align*}
	\lambda =\frac{\sqrt{k}+1}{2},~\Lambda =\bar{m}+\bar{B}+\frac{2}{\tau}+3,~l=\underline{m}+2,~L=\bar{m}+2,
\end{align*}
and the value of $\bar{m}$ will be given later. $V$ is a bounded and open subset of $U$, and we have
\begin{align*}
	\partial V=\{(g,m)\in U,&\lambda \le g\le\Lambda,l\le m\le L,~\mbox{and}~\exists~x\in[\varepsilon _0,1],\\ 
	&s.t:g(x)=\lambda~\mbox{or}~g(x)=\Lambda~\mbox{or}~m(x)=l ~\mbox{or}~m(x)=L\}.	
\end{align*} 

For given $\tau>0$, $\bar{B}\ge 0$ and $\underline{m}>1$, there exists a unique $\eta^\ast>1$ such that 
\begin{align}\label{yy42}
	w(\eta^\ast)-w(\underline{m}+3)=\frac{4(1-\varepsilon_0)}{\tau}-8{\rm ln}\varepsilon_0+w_k(\underline{m} +\bar{B}+\frac{2}{\tau}+5)-\frac{k}{2},
\end{align}
where the functions $w(h):={\rm ln}h+\frac{1}{2h^2}$ and $w_k(h):={\rm ln}h+\frac{k}{2h^2}$. We will prove there exists a pair of subsonic solutions to equation \eqref{yy41} when $\eta_0\ge\eta^\ast$. For any $\eta_0\ge\eta^\ast$, it yields
\begin{align}\label{yy43}
	w(\eta_0)-w(\underline{m}+3)\ge \frac{4(1-\varepsilon_0)}{\tau}-8{\rm ln}\varepsilon_0+w_k(\underline{m} +\bar{B}+\frac{2}{\tau}+5)-\frac{k}{2}.	
\end{align}
Now we take $\bar{m}>\eta_0$ satisfying
\begin{align}\label{yy44}
	w(\bar{m})-w(\eta_0)\ge \frac{4(1-\varepsilon_0)}{\tau}-12{\rm ln}\varepsilon_0\ge -4{\rm ln}\varepsilon_0-\frac{1}{\tau}\int_{\varepsilon_0}^{1}\left(\frac{1}{\hat{g}}-\frac{1}{\hat{m}}\right)dr,
\end{align}
for any $(\hat{g},\hat{m})\in \bar{V}$. There exists a unique $\hat{\eta}_1\in[\underline{m}+3,\bar{m}]$ such that
\begin{align}\label{yy45}
	w(\hat{\eta}_1)-w(\eta_0)=-4{\rm ln}\varepsilon_0-\frac{1}{\tau}\int_{\varepsilon_0}^{1}\left(\frac{1}{\hat{g}}-\frac{1}{\hat{m}}\right)dr.
\end{align}

Now define the operator $\Gamma:\bar{V}\to U$, that is, $(\hat{g},\hat{m})\mapsto (g_k,m_k)$ by solving the following linearized equation
\begin{equation}
	\left\{
	\begin{aligned}\label{yy46}
		&\left[r^2 \left( \frac{1}{\hat{g}}-\frac{k}{\hat{g}^3} \right) (g_k)_r \right]_r-\frac{r^2}{\tau \hat{g}^2}( g_k )_r=\hat{g}-\hat{m}-B(r)+2-\frac{2r}{\tau \hat{g}},\\
		&\left[r^2 \left( \frac{1}{\hat{m}}-\frac{1}{\hat{m}^3} \right) (m_k)_r \right]_r+\frac{r^2}{\tau \hat{m}^2} ( m_k )_r=\hat{m}+B(r)-\hat{g}+2+\frac{2r}{\tau \hat{m}},\\
		&g_k(\varepsilon _0)=g_k(1)=1,m_k(\varepsilon_0)=\eta _0,m_k(1)=\hat{\eta}_1,
	\end{aligned}
	\right.
\end{equation}
where $\hat{\eta}_1$ is defined in \eqref{yy45}. According to the theory of linear elliptic equations, $\Gamma:\bar{V}\to U$ is a compact and continuous operator. Let $\hat{\varphi}=\left (\hat{g},\hat{m}\right)$ and $\varphi=\left (g_k,m_k\right)=\Gamma \hat{\varphi}$.

Set $\Phi: \bar{V}\times [0,1]\to U$ is binary compact and continuous operator with $\Phi (\hat{\varphi},\gamma)=\hat{\varphi}-\gamma \Gamma \hat{\varphi}=\hat{\varphi}-\gamma\varphi$. Obviously, if $\Phi(\hat{\varphi},1)=0$, then $\hat{\varphi}$ is a solution of \eqref{yy41}. We take $q=(1,\eta_0)\in V$ and let $p(\gamma)=(1-\gamma)q$, where $\gamma \in [0,1]$. If $p(\gamma)\notin \Phi (\partial V,\gamma)$ for any $\gamma \in [0,1]$, then according to the topological degree theory, it holds
\begin{align*}
	deg\left ( \Phi(\cdot,1),V,0\right )=deg\left ( \Phi(\cdot,0),V,q\right )=deg\left ( id,V,q\right )=1,	
\end{align*}  
thus $\Phi(\hat{\varphi},1)=0$ gives a solution $\hat{\varphi} \in V$.

Next, we only need to prove that $p(\gamma)\notin \Phi (\partial V,\gamma)$ for any $\gamma \in [0,1]$. Assume there exists $\gamma \in [0,1]$ and $\hat{\varphi}\in\partial V$ such that
\begin{align*}
	p(\gamma)=(1-\gamma)q=\Phi (\hat{\varphi},\gamma)=\hat{\varphi}-\gamma \varphi.	
\end{align*}
We will show a contradition. If $\gamma=0$, we have $\hat{\varphi}=q$, which is impossible due to $q\in V, \hat{\varphi}\in \partial V$ and $V$ is an open subset of $U$. If $\gamma\in (0,1]$, we have
\begin{align*}
	\varphi=\frac{1}{\gamma}\hat{\varphi}-\frac{1-\gamma}{\gamma}q,
\end{align*}
that is, $(g_k)_r=\frac{1}{\gamma}\hat{g}_r$, $(m_k)_r=\frac{1}{\gamma}\hat{m}_r$, then by \eqref{yy46} we derive
\begin{equation}
	\left\{
	\begin{aligned}\label{yy47}
		&\left[r^2 \left( \frac{1}{\hat{g}}-\frac{k}{\hat{g}^3} \right)\hat{g}_r \right]_r-\frac{r^2}{\tau \hat{g}^2}\hat{g}_r=\gamma \left( \hat{g}-\hat{m}-B(r)+2-\frac{2r}{\tau \hat{g}} \right),\\
		&\left[r^2 \left( \frac{1}{\hat{m}}-\frac{1}{\hat{m}^3} \right)\hat{m}_r \right]_r+\frac{r^2}{\tau \hat{m}^2}\hat{m}_r=\gamma \left( \hat{m}+B(r)-\hat{g}+2+\frac{2r}{\tau \hat{m}} \right),\\
		&\hat{g}(\varepsilon _0)=\hat{g}(1)=1,\hat{m}(\varepsilon_0)=\eta _0,\hat{m}(1)=\eta_0+\gamma \left (\hat{\eta}_1-\eta_0\right)=:\tilde{\eta}_1,
	\end{aligned}
	\right.
\end{equation}
where $\tilde{\eta}_1$ is between $\eta_0$ and $\hat{\eta}_1$. Next, we prove that $\lambda<\hat{g}<\Lambda$ and $l<\hat{m}<L$.

First, the upper and lower bounds of $\hat{g}$ are to be estimated. Taking respectively $\omega_1=(\hat{g}-1)^-$ and $\omega_2=\left(\hat{g}-\left(\bar{m}+\bar{B}+\frac{2}{\tau}+2\right)\right)^+$ as the test function to $\eqref{yy47}_1$, and utilizing the standard weak maximum principle, we can obtain that $\omega_1\le 0$ and $\omega_2 \le 0$, which gives 
\begin{align}\label{yy48}
	\lambda<1\le \hat{g}\le \bar{m}+\bar{B}+\frac{2}{\tau}+2<\Lambda.	
\end{align}
Here and after $(h)^-:=-{\rm min}\left\{h,0\right\}$, $(h)^+:={\rm max}\left\{h,0\right\}$. Next, we will carefully estimate the upper and lower bounds of $\hat{m}$. From \eqref{yy47}, we have
\begin{equation}
	\left\{
	\begin{aligned}\label{yy49}
		&\left [r^2\left (\frac{1}{\hat{g}}-\frac{k}{\hat{g}^3}\right)\hat{g}_r+\frac{r^2}{\tau \hat{g}}-2r\right]_r=\gamma (\hat{g}-\hat{m}-B(r))+(1-\gamma)\frac{2r}{\tau \hat{g}}+2(\gamma-1),\\
		&\left [r^2\left (\frac{1}{\hat{m}}-\frac{1}{\hat{m}^3}\right)\hat{m}_r-\frac{r^2}{\tau \hat{m}}-2r\right]_r=\gamma (\hat{m}+B(r)-\hat{g})-(1-\gamma)\frac{2r}{\tau \hat{m}}+2(\gamma-1).
	\end{aligned}
	\right.
\end{equation}
Adding $\eqref{yy49}_1$ to $\eqref{yy49}_2$, it yields
\begin{align*}
	&\left [r^2\left (\frac{1}{\hat{g}}-\frac{k}{\hat{g}^3}\right)\hat{g}_r+r^2\left (\frac{1}{\hat{m}}-\frac{1}{\hat{m}^3}\right)\hat{m}_r+\frac{r^2}{\tau \hat{g}}-\frac{r^2}{\tau \hat{m}}-4r\right]_r\nonumber\\
	=&2(1-\gamma)r\cdot\frac{1}{\tau}\left (\frac{1}{\hat{g}}-\frac{1}{\hat{m}}\right)+4(\gamma-1).	
\end{align*}
Integrating the above equation on $\left(\varepsilon _0,r\right)$ and dividing by $r^2$, then we obtain
\begin{align}\label{yy50}
	&\left (\frac{1}{\hat{g}}-\frac{k}{\hat{g}^3}\right)\hat{g}_r+\left (\frac{1}{\hat{m}}-\frac{1}{\hat{m}^3}\right)\hat{m}_r+\frac{1}{\tau \hat{g}}-\frac{1}{\tau \hat{m}}-\frac{4}{r}\nonumber\\
	=&\frac{c}{r^2}+\frac{1}{r^2} \int_{\varepsilon _0}^{r}\left [2(1-\gamma)s\cdot\frac{1}{\tau}\left (\frac{1}{\hat{g}(s)}-\frac{1}{\hat{m}(s)}\right)+4(\gamma-1)\right]ds.
\end{align}
Once again, integrating the above equation on $\left [\varepsilon _0,1\right ]$, by \eqref{yy45} we get 
\begin{align}\label{yy51}
	&c\int_{\varepsilon_0}^{1}\frac{1}{r^2}dr\nonumber\\
	=&[w(\tilde{\eta}_1)-w(\hat{\eta}_1)]-\int_{\varepsilon_0}^{1}\frac{1}{r^2} \int_{\varepsilon _0}^{r}\left [2(1-\gamma)s\cdot\frac{1}{\tau}\left (\frac{1}{\hat{g}(s)}-\frac{1}{\hat{m}(s)}\right)\right]dsdr-\int_{\varepsilon_0}^{1}\frac{1}{r^2}4(\gamma-1)dsdr\nonumber\\	
	=&:A_1+A_2+A_3.
\end{align}
It is easy to calculate that
\begin{align}\label{yy52}
	c\int_{\varepsilon_0}^{1}\frac{1}{r^2}dr=c \cdot \frac{1-\varepsilon _0}{\varepsilon _0}   	
\end{align}
and
\begin{align}\label{yy53}
	\left|A_1\right|&=\left|w(\tilde{\eta}_1)-w(\eta_0)-\left(w(\hat{\eta}_1)-w(\eta_0)\right)\right|\nonumber\\
	&\le \left|w(\hat{\eta}_1)-w(\eta_0)\right|\nonumber\\
	&=\left|\frac{1}{\tau}\int_{\varepsilon _0}^{1}\frac{1}{\hat{g}}-\frac{1}{\hat{m}}dr+4{\rm ln}\varepsilon_0\right|\nonumber\\
	&\le \frac{1-\varepsilon_0}{\tau}-4{\rm ln}\varepsilon _0,	
\end{align}
due to $\tilde{\eta}_1$ is between $\eta_0$ and $\hat{\eta}_1$. Also, we have
\begin{align}\label{yy54}
	\left|A_2\right|\le&\frac{1}{\tau}\int_{\varepsilon _0}^{1}\frac{1}{r^2}\cdot \left|\int_{\varepsilon_0}^{r}2(1-\gamma)s\left(\frac{1}{\hat{g}}-\frac{1}{\hat{m}}\right)ds\right|dr\nonumber\\
	=&\frac{1-\gamma}{\tau}\int_{\varepsilon_0}^{1}\frac{1}{r^2}(r^2-\varepsilon_0^2)dr\le\frac{1-\varepsilon_0}{\tau}
\end{align}
and
\begin{align}\label{yy55}
	\left|A_3\right|\le\int_{\varepsilon _0}^{1}\frac{1}{r^2}\int_{\varepsilon_0}^{r}4dsdr=4\int_{\varepsilon _0}^{1}\frac{1}{r^2}(r-\varepsilon _0)dr\le4\int_{\varepsilon _0}^{1}\frac{1}{r}dr=-4{\rm ln}\varepsilon _0.
\end{align}
Combining \eqref{yy51}-\eqref{yy55}, we get
\begin{align}\label{yy56}
	\left|c\right|\le&\frac{\varepsilon _0}{1-\varepsilon _0}\left|A_1\right|+\frac{\varepsilon _0}{1-\varepsilon _0}\left|A_2\right|+\frac{\varepsilon _0}{1-\varepsilon _0}\left|A_3\right|\nonumber\\
	\le&\frac{\varepsilon _0}{1-\varepsilon _0}\cdot\left(\frac{1-\varepsilon _0}{\tau}-4{\rm ln}\varepsilon _0\right)+\frac{\varepsilon _0}{1-\varepsilon _0}\cdot\frac{1-\varepsilon _0}{\tau}-\frac{\varepsilon _0}{1-\varepsilon _0}\cdot 4{\rm ln}\varepsilon _0\nonumber\\
	=&\frac{2\varepsilon _0}{\tau}-\frac{8\varepsilon _0}{1-\varepsilon _0}{\rm ln}\varepsilon _0.
\end{align}

Now we are ready to estimate the upper bound of $\hat{m}$. Integrating \eqref{yy50} on $(\varepsilon_0,r)$, we have
\begin{align}\label{yy57}
	&w_k(\hat{g})-\frac{k}{2}+w(\hat{m})-w(\eta_0)+\frac{1}{\tau}\int_{\varepsilon_0}^{r}\frac{1}{\hat{g}(\xi)}-\frac{1}{\hat{m}(\xi)}d\xi-4{\rm ln}r+4{\rm ln}\varepsilon_0\nonumber\\
	=&c\left(\frac{1}{\varepsilon_0}-\frac{1}{r}\right)+\int_{\varepsilon_0}^{r}\frac{1}{\xi^2}\int_{\varepsilon_0}^{\xi}2(1-\gamma)s\cdot\frac{1}{\tau}\left(\frac{1}{\hat{g}(s)}-\frac{1}{\hat{m}(s)}\right)dsd\xi+\int_{\varepsilon_0}^{r}\frac{1}{\xi^2}\int_{\varepsilon_0}^{\xi}4(\gamma-1)dsd\xi,	
\end{align} 
and then
\begin{align}\label{yy58}
	w(\hat{m})-w(\eta_0)=&\frac{k}{2}-w_k(\hat{g})-\frac{1}{\tau}\int_{\varepsilon_0}^{r}\left(\frac{1}{\hat{g}(\xi)}-\frac{1}{\hat{m}(\xi)}\right)d\xi+4{\rm ln}r-4{\rm ln}\varepsilon_0+c\left(\frac{1}{\varepsilon_0}-\frac{1}{r}\right)\nonumber\\ &+\int_{\varepsilon_0}^{r}\frac{1}{\xi^2}\int_{\varepsilon_0}^{\xi}2(1-\gamma)s\cdot \frac{1}{\tau}\left(\frac{1}{\hat{g}(s)}-\frac{1}{\hat{m}(s)}\right)dsd\xi+\int_{\varepsilon_0}^{r}\frac{1}{\xi^2}\int_{\varepsilon_0}^{\xi}4(\gamma-1)dsd\xi. 	
\end{align}
Through a series of simple calculations, it yields
\begin{align}\label{yy59}
	\left|\frac{1}{\tau}\int_{\varepsilon_0}^{r}\frac{1}{\hat{g}(\xi)}-\frac{1}{\hat{m}(\xi)}d\xi\right|\le\frac{1-\varepsilon_0}{\tau},
\end{align}

\begin{align}\label{yy60}
	\left|\int_{\varepsilon_0}^{r}\frac{1}{\xi^2}\int_{\varepsilon_0}^{\xi}2(1-\gamma)s\cdot \frac{1}{\tau}\left(\frac{1}{\hat{g}(s)}-\frac{1}{\hat{m}(s)}\right)dsd\xi\right| \le&\frac{1-\gamma}{\tau}\int_{\varepsilon_0}^{r}\frac{1}{\xi^2}\int_{\varepsilon_0}^{\xi}2s~dsd\xi\nonumber\\
	=&\frac{1-\gamma}{\tau}\int_{\varepsilon_0}^{r}\frac{1}{\xi^2}(\xi^2-\varepsilon _0^2)d\xi\nonumber\\
	\le&\frac{1-\varepsilon_0}{\tau}
\end{align}
and
\begin{align}\label{yy61}
	\int_{\varepsilon_0}^{r}\frac{1}{\xi^2}\int_{\varepsilon_0}^{\xi}4(\gamma-1)dsd\xi \le 0.	
\end{align}
Notice that $w_k(\hat{g})\ge \frac{k}{2}$ due to $\hat{g}\ge 1$ on $[\varepsilon_0,1]$. Substituting \eqref{yy56} and \eqref{yy59}-\eqref{yy61} into \eqref{yy58}, we derive
\begin{align}\label{yy62}
	w(\hat{m})-w(\eta_0)\le&\frac{1-\varepsilon_0}{\tau}-4{\rm ln}\varepsilon_0+\left(\frac{2\varepsilon_0}{\tau}-\frac{8\varepsilon_0}{1-\varepsilon_0}{\rm ln}\varepsilon_0\right)\left(\frac{1-\varepsilon_0}{\varepsilon_0}\right)+\frac{1-\varepsilon_0}{\tau}\nonumber\\
	\le&\frac{4(1-\varepsilon_0)}{\tau}-12{\rm ln}\varepsilon_0.	
\end{align}
So we obtain that $\hat{m}\le \bar{m}<L$ by \eqref{yy44} and \eqref{yy62}.

Next, we estimate the lower bound of $\hat{m}$. From \eqref{yy57}, we have
\begin{align}\label{yy63}
	w_k(\hat{g})-\frac{k}{2}=&w(\eta_0)-w(\hat{m})-\frac{1}{\tau}\int_{\varepsilon_0}^{r}\left(\frac{1}{\hat{g}(\xi)}-\frac{1}{\hat{m}(\xi)}\right)d\xi+4{\rm ln}r-4{\rm ln}\varepsilon_0+c\left(\frac{1}{\varepsilon_0}-\frac{1}{r}\right)\nonumber\\
	&+\int_{\varepsilon_0}^{r}\frac{1}{\xi^2}\int_{\varepsilon_0}^{\xi}2(1-\gamma)s\cdot \frac{1}{\tau}\left(\frac{1}{\hat{g}(s)}-\frac{1}{\hat{m}(s)}\right)dsd\xi+\int_{\varepsilon_0}^{r}\frac{1}{\xi^2}\int_{\varepsilon_0}^{\xi}4(\gamma-1)dsd\xi.
\end{align}
Since
\begin{align}\label{yy64}
	\left|\int_{\varepsilon_0}^{r}\frac{1}{\xi^2}\int_{\varepsilon_0}^{\xi}4(\gamma-1)dsd\xi\right| \le&\int_{\varepsilon_0}^{r}\frac{1}{\xi^2}\int_{\varepsilon_0}^{\xi}4~dsd\xi\nonumber\\
	=&\int_{\varepsilon_0}^{r}\frac{1}{\xi^2}4(\xi-\varepsilon _0)d\xi\nonumber\\
	\le&4\int_{\varepsilon_0}^{r}\frac{1}{\xi}d\xi\nonumber\\
	=&4{\rm ln}r-4{\rm ln}\varepsilon_0.	
\end{align}
Thus, plugging \eqref{yy56}, \eqref{yy59}-\eqref{yy60} and \eqref{yy64} into \eqref{yy63} gives
\begin{align}\label{yy65}
	&w_k(\hat{g})-\frac{k}{2}\nonumber\\
	\ge &w(\eta_0)-w(\hat{m})-\frac{1-\varepsilon _0}{\tau}-\left(\frac{2\varepsilon_0}{\tau}-\frac{8\varepsilon_0}{1-\varepsilon_0}{\rm ln}\varepsilon_0\right ) \left(\frac{1-\varepsilon_0}{\varepsilon_0}\right)-\frac{1-\varepsilon _0}{\tau}\nonumber\\
	\ge &w(\eta_0)-w(\hat{m})-\frac{4(1-\varepsilon _0)}{\tau}+8{\rm ln}\varepsilon_0.	
\end{align}
Define $\hat{\Omega}:=\left\{r\in [\varepsilon_0,1],\hat{m}\le \underline{m}+3\right\}$. Since $\eta_0\ge \eta^\ast$, we have 
\begin{align}\label{yy66}
	w(\eta_0)-w(\hat{m})\ge w(\eta^\ast)-w(\underline{m}+3)\,\,\,\,\,\,\mbox{in}\,\,\,\hat{\Omega}. 
\end{align}
Thus
\begin{align}\label{yy67}
	w_k(\hat{g})-\frac{k}{2}&\ge w(\eta^\ast)-w(\underline{m}+3)-\frac{4(1-\varepsilon _0)}{\tau}+8{\rm ln}\varepsilon_0\,\,\,\,\,\,\mbox{in}\,\,\,\hat{\Omega}. 	
\end{align}
By \eqref{yy42} and \eqref{yy67}, we obtain
\begin{align}\label{yy68}
	w_k(\hat{g})-\frac{k}{2}\ge w_k(\underline{m}+\bar{B}+\frac{2}{\tau}+5)-\frac{k}{2}\,\,\,\,\,\,\mbox{in}\,\,\,\hat{\Omega}. 	
\end{align}
Thus it can be deduced
\begin{align}\label{yy69}
	\hat{g}\ge \underline{m}+\bar{B}+\frac{2}{\tau}+5\,\,\,\,\,\,\mbox{in}\,\,\,\hat{\Omega}.
\end{align}
To prove $\hat{m}\ge \underline{m}+3>l$, we take $(\hat{m}-(\underline{m}+3))^-$ as the test function. Multiplying $(\hat{m}-(\underline{m}+3))^-$ for $\eqref{yy47}_2$ and integrating it by parts on $[\varepsilon_0,1]$, we obtain
\begin{align}\label{yy70}
	&\int_{\varepsilon_0}^{1}r^2\left(\frac{1}{\hat{m}}-\frac{1}{\hat{m}^3}\right)\left|((\hat{m}-(\underline{m}+3))^-)_r\right|^2dr-\int_{\varepsilon_0}^{1}\frac{r^2}{\tau \hat{m}^2}(\hat{m}-(\underline{m}+3))^-\cdot (\hat{m}-(\underline{m}+3))_r^-dr\nonumber\\
	=&\gamma \int_{\varepsilon_0}^{1}\left(\hat{m}+B(r)-\hat{g}+2+\frac{2r}{\tau \hat{m}}\right) \cdot (\hat{m}-(\underline{m}+3))^- dr\nonumber\\
	=&\gamma \int_{\hat{\Omega}}\left(\hat{m}+B(r)-\hat{g}+2+\frac{2r}{\tau \hat{m}}\right) \cdot (\hat{m}-(\underline{m}+3))^- dr\nonumber\\
	\le&\gamma\int_{\hat{\Omega}}\left(\hat{m}+B(r)-\left(\underline{m}+\bar{B}+\frac{2}{\tau}+5\right)+2+\frac{2r}{\tau \hat{m}}\right)\cdot(\hat{m}-(\underline{m}+3))^- dr\nonumber\\
	\le&0.	
\end{align}
Owing to $\frac{1}{\hat{m}}-\frac{1}{\hat{m}^3}>0$ and $\left\|\frac{r^2}{\tau\hat{m}^2}\right\|_{L^\infty[\varepsilon_0,1]}$ is bounded, using the weak maximum principle in \cite[Theorem 8.1]{GT2001}, we ultimately obtain that $(\hat{m}-(\underline{m}+3))^-\equiv 0$, that is, $\hat{m}\ge \underline{m}+3>l$.

In summary, we have proved that $\lambda<\hat{g}<\Lambda$ and $l<\hat{m}<L$ for $\eta_0\ge \eta^\ast$, which is contradictory with $(\hat{g},\hat{m})\in \partial V$. Utilizing the standard regularity theory and the discussions above, we can infer to $(g_k,m_k)\in W^{2,\infty}(\varepsilon_0,1)\times W^{2,\infty}(\varepsilon_0,1)$ and $m\ge \underline{m}+3$.

Finally, we apply the same method as in \cite[Lemma 2.3]{LMZZ2017} to obtain
\begin{align}\label{yy71}
	g_k(r)\ge 1+\bar{\nu} sin(\pi \cdot\frac{r-\varepsilon_0}{1-\varepsilon_0})>1,	
\end{align}
where $\bar{\nu}>0$ is a small constant and independent of $k$.

This completes the proof of Lemma \ref{le1}.
\end{proof}
Now, we are ready to prove Theorem \ref{th1}.\\
{\bf The proof of Theorem \ref{th1}}.~~We apply the compactness method as in \cite{LMZZ2017} to get the solution of \eqref{eq15}. Assume that $(g_k,m_k)$ is the solution of \eqref{yy41}. Multiplying $(g_k-1)$ for the equation $\eqref{yy41}_1$ and integrating it by parts on $[\varepsilon_0,1]$, we have
\begin{align*}
	&(1-k^2)\int_{\varepsilon_0}^{1}r^2\frac{|(g_k)_r|^2}{(g_k)^3}dr+\frac{4}{9}\int_{\varepsilon_0}^{1}r^2\frac{g_k+1}{(g_k)^3}|((g_k-1)^\frac{3}{2})_r|^2dr\\
	&\,+\frac{1}{\tau}\int_{\varepsilon_0}^{1}\frac{r^2(g_k)_r}{g_k}dr+\int_{\varepsilon_0}^{1}(g_k-m_k-B+2)(g_k-1)dr=0.	
\end{align*} 
Applying the standard energy estimate and the uniform boundedness of $g_k$ and $m_k$ in $L^\infty(\varepsilon_0,1)$, we derive
\begin{align*}
	\|((g_k-1)^\frac{3}{2})_r\|_{L^2[\varepsilon_0,1]}\le C,	
\end{align*}
here and below $C$ denotes constant independent of $k$. Owing to $((g_k-1)^2)_r=\frac{4}{3}(g_k-1)^\frac{1}{2}((g_k-1)^\frac{3}{2})_r $, according to the boundedness of $g_k$, we have 
\begin{align}\label{yy72}
	\|(g_k-1)^2\|_{H_0^1[\varepsilon_0,1]}\le C.	
\end{align}

Now, multiplying $\eqref{yy41}_2$ by $(m_k-m_{\theta_k})$, where $m_{\theta_k}(r)=\eta_0+r(m_k(1)-\eta_0)$, and integrating it by parts on $[\varepsilon_0,1]$, we arrive
\begin{align*}
	\int_{\varepsilon_0}^{1}r^2\frac{m_k^2-1}{m_k^3}|(m_k)_r|^2dr=&\int_{\varepsilon_0}^{1}r^2\frac{m_k^2-1}{m_k^3}(m_k)_r(m_{\theta_k})_rdr+\frac{1}{\tau}\int_{\varepsilon_0}^{1}\frac{1}{m_k}(m_k-m_{\theta_k})_rdr\\
	&\,-\int_{\varepsilon_0}^{1}(m_k+B-g_k+2)(m_k-m_{\theta_k})dr.	
\end{align*}
Based on the boundedness of $g_k$ and $m_k$ again, it can be obtained that $\|(m_k)_r\|_{L^2[\varepsilon_0,1]}\le C$, and then $\|m_k\|_{H^1[\varepsilon_0,1]}\le C$. Consequently, we get 
\begin{align}\label{yy73}
	&(g_k-1)^2\rightharpoonup (g-1)^2~~\mbox{weakly~in}~~H_0^1[\varepsilon_0,1]~~\mbox{for}~~k\to 1,\\
	&m_k\rightharpoonup m~~\mbox{weakly~in}~~H^1[\varepsilon_0,1]~~\mbox{for}~~k\to 1	
\end{align}
and
\begin{align}\label{yy74}
	||(g-1)^2||_{H_0^1[\varepsilon_0,1]}\le C,~~||m||_{H^1[\varepsilon_0,1]}\le C.	
\end{align}

Using the similar method as in \cite{LMZZ2017}, we can prove that $(g,m)\in C^{\frac{1}{2}}[\varepsilon_0,1]\times W^{2,\infty}(\varepsilon_0,1)$ and $(g,m)$ is the weak solution of the system \eqref{eq15}. At the same time, since $H^1(\varepsilon_0,1)$ is the compact embedding of $C[\varepsilon_0,1]$, so we obtain that $g>1$ in $(\varepsilon_0,1)$, $m\ge \underline{m}+3$ in $[\varepsilon_0,1]$ and the conditions \eqref{c3}-\eqref{y9} all hold by \eqref{yy71} and $m_k\ge \underline{m}+3$. Therefore, we complete the proof of Theorem \ref{th1}.\hfill $\Box$ 

%%%%%%%%%%%%%%%%%%%%%%%%%%%%%%%%%%%%%%%%%%%%%%%%%%%%%%%%%%%%%%%%%%%%%%%%%%%%%%%%%%%%%%%%%%%%%%%%%%

\section{The ill-posedness of subsonic solution}

In this section, we prove the ill-posedness of subsonic solution to \eqref{eq10}.\\
{\bf The proof of Theorem \ref{th2}}.~~According to the model \eqref{eq10}, it yields
\begin{align}\label{eq18}
\left(\frac{1}{m}-\frac{1}{m^3}\right)m_r=-\left(\frac{1}{g}-\frac{1}{g^3}\right)g_r-\frac{1}{\tau}\left(\frac{1}{g}-\frac{1}{m}\right)+\frac{4}{r}.	
\end{align}
Integrating \eqref{eq18} on $(\varepsilon_0,r)$, we derive
\begin{align}\label{eq19}
w(m)-w(\eta _0)=-w(g)+\frac{1}{2}-\frac{1}{\tau}\int_{\varepsilon_0}^{r}\left(\frac{1}{g}-\frac{1}{m}\right)dr+4{\rm ln}r-4{\rm ln}\varepsilon _0.		
\end{align}
Due to $\left|\frac{1}{\tau}\int_{\varepsilon_0}^{r}\left(\frac{1}{g}-\frac{1}{m}\right)dr\right|\le \frac{1}{\tau}$, it holds
\begin{align}\label{eq20}
w(m)-w(\eta _0)\le -w(g)+\frac{1}{2}+\frac{1}{\tau}-4{\rm ln}\varepsilon _0.
\end{align}
Thus for any fixed $\bar{\eta}>1$, it should hold for all $\eta_0<\bar{\eta}$
\begin{align*}
w(m)-w(\eta_0)>w(1)-w(\bar{\eta})>-w(\bar{\eta}).
\end{align*}
Hence, it is true 
\begin{align}\label{g1}
-w(g)+\frac{1}{2}+\frac{1}{\tau}-4{\rm ln}\varepsilon_0>-w(\bar{\eta}),	
\end{align}
which gives a restriction on the upper bound of $w(g)$, and thus a restriction on the upper bound of $g$. That is, $w(g)<w(\bar{\eta})+\frac{1}{2}+\frac{1}{\tau}-4{\rm ln}\varepsilon_0$. This gives the idea of the proof. Let $\bar{g}=\bar{g}(\bar{\eta})>1$ such that $w(\bar{g})=w(\bar{\eta})+\frac{1}{2}+\frac{1}{\tau}-4{\rm ln}\varepsilon_0$. So it holds that $g<\bar{g}$, and then $w(g)<w(\bar{g})$, for any $r\in[\varepsilon_0,1]$. 

Substituting $w(h)={\rm ln} h+\frac{1}{2h^2}$ into the model $\eqref{eq15}_1$, it becomes
\begin{equation}
	\left\{
	\begin{aligned}\label{equa4}
		&\left[r^2w_r(g(r))\right]_r+\left(\frac{r^2}{\tau g}\right)_r=g-m-B(r)+2,\\
		&w(g(\varepsilon _0))=w(g(1))=\frac{1}{2}.
	\end{aligned}
	\right.
\end{equation}

Assume that $p$ is the maximum point of $w(g(r))$ and $p$ lies in the right half interval of $[\varepsilon_0, 1]$ without loss of generality, that is, $p\in \left(\varepsilon_0+\frac{1-\varepsilon_0}{2}, 1\right)$. Taking any point $q\in \left(\varepsilon_0,\varepsilon_0+\frac{1-\varepsilon_0}{4}\right)$, we have $p-q>\frac{1-\varepsilon_0}{4}$ and $w_r(g(p))=0$. Due to $w\in C^{1,\frac{1}{2}}[\varepsilon_0,1]\cap W^{2,1}[\varepsilon_0,1]$ (see \cite{LMZZ2017}), then for any $q\in \left(\varepsilon_0,\varepsilon_0+\frac{1-\varepsilon_0}{4}\right)$, we integrate $\eqref{equa4}_1$ on $(q,p)$ to have
\begin{align}\label{eq22}
w_r(g(q))&=-\frac{1}{q^2}\int_{q}^{p}(g-m-B(s)+2)ds+\frac{1}{\tau q^2}\left(\frac{p^2}{g(p)}-\frac{q^2}{g(q)}\right)\nonumber\\
&\ge \frac{1}{q^2}\int_{q}^{p}(B(s)+m-g-2)ds-\frac{1}{\tau}\nonumber\\
&\ge \frac{1}{q^2}\int_{q}^{p}(B(s)-(\bar{g}+2))ds-\frac{1}{\tau}\nonumber\\
&\ge \frac{1}{q^2}\left(\int_{q}^{p}B(s)ds-(\bar{g}+2)\right)-\frac{1}{\tau}\nonumber\\
&\ge \frac{1}{q^2}\left(\frac{1-\varepsilon_0}{4}B^\ast-(\bar{g}+2)\right)-\frac{1}{\tau},
\end{align}
where $\underset{r\in(\alpha,\beta)}{{\rm inf}}B=B^\ast$, and $\alpha=\varepsilon_0+\frac{1-\varepsilon_0}{4},\beta=\varepsilon_0+\frac{3(1-\varepsilon_0)}{4}$. Based on this, we have
\begin{align}\label{eq23}
&w(g(\varepsilon_0+\frac{1-\varepsilon_0}{4}))-w(g(\varepsilon_0))\nonumber\\
=&\int_{\varepsilon_0}^{\varepsilon_0+\frac{1-\varepsilon_0}{4}}w_r(g(s))ds\nonumber\\
\ge& \int_{\varepsilon_0}^{\varepsilon_0+\frac{1-\varepsilon_0}{4}}\left[\frac{1}{s^2}\left(\frac{1-\varepsilon_0}{4}B^\ast-(c+2)\right)-\frac{1}{\tau}\right]ds\nonumber\\
=&\frac{1-\varepsilon_0}{(3\varepsilon_0+1)\varepsilon_0}\left(\frac{1-\varepsilon_0}{4}B^\ast-(\bar{g}+2)\right)-\frac{1-\varepsilon_0}{4\tau}.
\end{align}
Taking 
\begin{align}\label{g4}
B^\ast=B^\ast(\bar{\eta}):=\frac{(3\varepsilon_0+1)4\varepsilon_0}{(1-\varepsilon_0)^2}w(\bar{g})+\frac{(3\varepsilon_0+1)\varepsilon_0}{(1-\varepsilon_0)\tau}+\frac{4(\bar{g}+2)}{1-\varepsilon_0}.
\end{align}
Then for $\underset{r\in(\alpha,\beta)}{{\rm inf}}B\ge B^\ast(\bar{\eta})$, we derive
\begin{align}\label{ww1}
w(g(\varepsilon_0+\frac{1-\varepsilon_0}{4}))-w(g(\varepsilon_0))\ge w(\bar{g}),
\end{align}
which gives a contradiction with $g<\bar{g}(\bar{\eta})$ and $w(g)<w(\bar{g})$. This completes the proof of Theorem \ref{th2}.\hfill $\Box$

%%%%%%%%%%%%%%%%%%%%%%%%%%%%%%%%%%%%%%%%%%%%%%%%%%%%%%%%%%%%%%%%%%%%%%%%%%%%%%%%%%%%%%%%%%%%%%%%%%

\section{The uniqueness of subsonic solution}

\subsection{The uniqueness of subsonic solution when $\tau=\infty$}

In this subsection, we use the direct energy method to prove the uniqueness of subsonic solution when the relaxation time $\tau=\infty$, namely, the pure Euler-Poisson case. The model becomes
\begin{equation}
	\left\{
	\begin{aligned}\label{yy1}
		&\left(\frac{1}{g}-\frac{1}{g^3}\right)g_r-\frac{2}{r}=E,\\
		&\left(\frac{1}{m}-\frac{1}{m^3}\right)m_r-\frac{2}{r}=-E,\\
		&\left(r^2E\right)_r=g-m-B(r),\\
		&g(\varepsilon_0)=g(1)=1,m(\varepsilon_0)=\eta_0,
	\end{aligned}
	\right.
\end{equation}
and its equivalent equation is
\begin{equation}
	\left\{
	\begin{aligned}\label{yy2}
		&\left[r^2\left(\frac{1}{g}-\frac{1}{g^3}\right)g_r\right]_r =g-m-B(r)+2,\\
		&\left(\frac{1}{m}-\frac{1}{m^3}\right)m_r-\frac{2}{r}=-E,\\
		&g(\varepsilon_0)=g(1)=1,m(\varepsilon_0)=\eta_0.
	\end{aligned}
	\right.
\end{equation}
{\bf The proof of Theorem \ref{th3} for the case of $\tau=\infty$}.
Assume that $(g_1,m_1)$ and $(g_2,m_2)$ are two pairs of solutions to the equation \eqref{yy1}, it holds by \eqref{yy2}
\begin{align}\label{yy3}
	(r^2(w(g_1)-w(g_2))_r)_r=(g_1-g_2)-(m_1-m_2).	
\end{align}

Taking $(w(g_1)-w(g_2))^+$ as the test function, multiplying \eqref{yy2} by $(w(g_1)-w(g_2))^+$ and integrating it by parts on $[\varepsilon_0,1]$, we derive
\begin{align}\label{yy4}
&-\int_{\varepsilon_0}^{1}r^2\left|(w(g_1)-w(g_2))^{+}_r\right|^2dr\nonumber\\
=&\int_{\varepsilon_0}^{1}(g_1-g_2)\cdot (w(g_1)-w(g_2))^{+}dr-\int_{\varepsilon_0}^{1}(m_1-m_2)\cdot (w(g_1)-w(g_2))^{+}dr\nonumber\\
=&:B_1+B_2.	
\end{align}
According to the monotonicity of function $w$, it holds that $B_1\ge 0$. By $\eqref{yy1}_1$ and $\eqref{yy1}_2$, we have 
\begin{align*}
	(w(g))_r-\frac{4}{r}=-(w(m))_r.
\end{align*}
Integrating the above equation on $(\varepsilon_0,r)$ and collecting $(g_1,m_1)$ and $(g_2,m_2)$, we get
\begin{align}\label{yy5}
	w(g_1)-w(g_2)=w(m_2)-w(m_1).	
\end{align}
For $w(g_1)\ge w(g_2)$, we naturally have $w(m_2)\ge w(m_1)$. Furthermore, based on the monotonicity of function $w$ again, it holds that $B_2\ge 0$. Hence, we have
\begin{align}\label{yy6} -\int_{\varepsilon_0}^{1}r^2\left|(w(g_1)-w(g_2))^{+}_r\right|^2dr\ge 0.
\end{align}
That is
\begin{align}\label{yy7}
	\left\|(w(g_1)-w(g_2))_r^+\right\|_{L^2[\varepsilon_0,1]}=0.
\end{align} 
Therefore, we ultimately obtain that $\left\|(w(g_1)-w(g_2))^+\right\|_{L^2[\varepsilon_0,1]}=0$ by the Poincar\'e inequality, that is, $g_1\le g_2$ on $[\varepsilon_0,1]$.

Similarly, taking the test function as $(w(g_1)-w(g_2))^-$ and multiplying \eqref{yy2} by $(w(g_1)-w(g_2))^-$, then performing the same procedure as above, we get $g_1\ge g_2$ on $[\varepsilon_0,1]$.

Thus, we have proved the uniqueness of subsonic solution to the system \eqref{eq2} when $\tau=\infty$.\hfill $\Box$

\subsection{The uniqueness of subsonic solution when $\frac{j}{\tau}\ll 1$}

In this subsection, we apply the method of exponential variation in \cite[Theorem 10.7]{GT2001} and make specific modifications to prove the uniqueness of subsonic solution when $\frac{j}{\tau}\ll 1$. Here, we take $j_1=j$ and $j_2=-j$ in \eqref{eq2}, then the following model is presented
\begin{equation}
	\left\{
	\begin{aligned}\label{yy8}
		&\left(\frac{1}{g}-\frac{j^2}{g^3}\right)g_r+\frac{j}{\tau}\frac{1}{g}-\frac{2}{r}=E,\\
		&\left(\frac{1}{m}-\frac{j^2}{m^3}\right)m_r-\frac{j}{\tau}\frac{1}{m}-\frac{2}{r}=-E,\\
		&\left(r^2E\right)_r=g-m-B(r),\\
		&g(\varepsilon_0)=g(1)=1,m(\varepsilon_0)=\eta_0,
	\end{aligned}
	\right.
\end{equation}
and its equivalent equation is
\begin{equation}
	\left\{
	\begin{aligned}\label{yy9}
		&\left[r^2\left(\frac{1}{g}-\frac{j^2}{g^3}\right)g_r\right]_r=g-m-B(r)+2-\frac{j}{\tau}\left(\frac{r^2}{g}\right)_r,\\
		&\left(\frac{1}{m}-\frac{j^2}{m^3}\right)m_r-\frac{j}{\tau}\frac{1}{m}-\frac{2}{r}=-E,\\
		&g(\varepsilon_0)=g(1)=1,m(\varepsilon_0)=\eta_0.
	\end{aligned}
	\right.
\end{equation} 

The method in Section 2 can still be applied for \eqref{yy8}, thus the subsonic solution also exists. Next, we focus on the uniqueness of subsonic solution for system \eqref{yy8}.\\
{\bf The proof of Theorem \ref{th3} for the case of $\frac{j}{\tau}\ll 1$}.
Suppose that $(g_1,m_1)$ and $(g_2,m_2)$ are two pairs of solutions of \eqref{yy8}, then from \eqref{yy9}, it holds 
\begin{align}\label{yy10}
	(r^2(\mathcal{F}(g_1)-\mathcal{F}(g_2))_r)_r=(g_1-g_2)-(m_1-m_2)-\frac{j}{\tau}\left(\frac{r^2}{g_1}-\frac{r^2}{g_2}\right)_r,	
\end{align}
where $\mathcal{F}(h):={\rm ln}h+\frac{j^2}{2h^2}$. We define $V:=\mathcal{F}(g_1)-\mathcal{F}(g_2)$ and take $\varphi_h:=\frac{V^+}{V^{+}+h}$ as the test function according to the idea of the comparison principle in \cite[Theorem 10.7]{GT2001}, where $h>0$. Then we can obtain $(\varphi_h)_r=\frac{h}{(V^++h)^2}V^+_r$ and $\left({\rm log}\left(1+\frac{V^+}{h}\right)\right)_r=\frac{V^+_r}{V^{+}+h}$. Multiplying the equation \eqref{yy10} by $\varphi_h$ and integrating by parts on $[\varepsilon_0,1]$, we obtain
\begin{align}\label{yy11}
&h\int_{\varepsilon_0}^{1}r^2\left|\left({\rm log}\left(1+\frac{V^+}{h}\right)\right)_r\right|^2dr+\int_{\varepsilon_0}^{1}(g_1-g_2)\cdot \frac{V^+}{V^++h}dr\nonumber\\
=&\frac{j}{\tau}\int_{\varepsilon_0}^{1}\frac{r^2(g_1-g_2)}{g_1g_2}\cdot\frac{h}{(V^++h)^2}\cdot V^+_rdr+\int_{\varepsilon_0}^{1}(m_1-m_2)\cdot \frac{V^+}{V^++h}dr\nonumber\\
=&:I_1+I_2.	
\end{align}

First we deal with $\int_{\varepsilon_0}^{1}(g_1-g_2)\cdot \frac{V^+}{V^++h}dr$. Assume that $(g_1-g_2)^+\not\equiv 0$, we have
\begin{align}\label{yy12}
	\underset{h\to0^+}{{\rm lim}}\int_{\varepsilon_0}^{1}(g_1-g_2)\cdot \frac{V^+}{V^++h}=\left\|(g_1-g_2)^+\right\|_{L^1[\varepsilon_0,1]}>M>0,
\end{align}
where $M$ denotes a positive constant. Here and below we use $\left\|\cdot\right\|_{L^p}$ to denote $\left\|\cdot\right\|_{L^p[\varepsilon_0,1]}$ for simplicity.

Next, it holds for $I_1$
\begin{align}\label{yy13}
	I_1\le \frac{j}{\tau}\int_{\varepsilon_0}^{1}r^2\left|\frac{h(g_1-g_2)}{V^++h}\right|\cdot\left|\left({\rm log}\left(1+\frac{V^+}{h}\right)\right)_r\right|dr.	
\end{align}
Because $g_1> 1, g_2> 1$ in $(\varepsilon_0,1)$ and $g_1=g_2=1$ at $r=\varepsilon_0$ and $r=1$, then it holds that $\underset{h\to0^+}{{\rm lim}} \left|\frac{g_1-g_2}{V^++h}\right|\to+\infty$ when $j=1$ for $r$ is near $\varepsilon_0$ or $1$. For this reason, the comparison principle in \cite[Theorem 10.7]{GT2001} cannot be directly applied in the following proof. Define the set $D:=\left\{r\in[\varepsilon_0,1]|\left|\frac{g_1-g_2}{V^++h}\right|\le \mathcal{C}\right\}$, where $\mathcal{C}>0$ is to be determined. Using the Young inequality, we have
\begin{align}\label{yy14}
	I_1&\le\frac{j}{\tau} h\int_{D}r^2\left|\frac{g_1-g_2}{V^++h}\right|\cdot\left|\left({\rm log}\left(1+\frac{V^+}{h}\right)\right)_r\right|dr\nonumber\\
	&\,\,\,\,\,\,\,\,\,\,+\frac{j}{\tau}\int_{D^c}r^2h\left|\frac{g_1-g_2}{V^++h}\right|\cdot\left|\left({\rm log}\left(1+\frac{V^+}{h}\right)\right)_r\right|dr\nonumber\\
	&\le\frac{j}{\tau} h\mathcal{C}\int_{D}r^2\left|\left({\rm log}\left(1+\frac{V^+}{h}\right)\right)_r\right|dr\nonumber\\
	&\,\,\,\,\,\,\,\,\,\,+\frac{j}{\tau}\int_{D^c}r^2h\cdot\left[\frac{1}{2\varepsilon}\left(\frac{g_1-g_2}{V^++h}\right)^2+\frac{\varepsilon}{2}\left|\left({\rm log}\left(1+\frac{V^+}{h}\right)\right)_r\right|^2\right]dr,
\end{align}
where $\varepsilon$ is a parameter, and $D^c=[\varepsilon_0,1]/D$. Taking $\varepsilon=\frac{\tau}{j}$, it follows from \eqref{yy14} that
\begin{align}\label{yy15}
	I_1&\le \frac{j}{\tau}h\mathcal{C}\int_{\varepsilon_0}^{1}r^2\left|\left({\rm log}\left(1+\frac{V^+}{h}\right)\right)_r\right|dr+\frac{j^2}{2\tau^2}\int_{D^c}r^2h\cdot\left(\frac{g_1-g_2}{V^++h}\right)^2dr\nonumber\\
	&\,\,\,\,\,\,\,\,\,\,+\frac{h}{2}\int_{\varepsilon_0}^{1}r^2\left|\left({\rm log}\left(1+\frac{V^+}{h}\right)\right)_r\right|^2dr\nonumber\\
	&\le \frac{j}{\tau} h\mathcal{C}\int_{\varepsilon_0}^{1}r^2\left|\left({\rm log}\left(1+\frac{V^+}{h}\right)\right)_r\right|dr+\frac{j^2}{2\tau^2}\int_{D^c}r^2\left|\frac{(g_1-g_2)^2}{V^++h}\right|dr\nonumber\\
	&\,\,\,\,\,\,\,\,\,\,+\frac{h}{2}\int_{\varepsilon_0}^{1}r^2\left|\left({\rm log}\left(1+\frac{V^+}{h}\right)\right)_r\right|^2dr.
\end{align}
According to the Taylor expansion, we have $\mathcal{F}(g_i)=\frac{j^2}{2}+\mathcal{F}'(1)(g_i-1)+\frac{{\mathcal{F}}''(1)}{2}(g_i-1)^2+o(g_i-1)^2$, where $i=1,2$, we can get
\begin{align}\label{yy16}
	\left|\frac{(g_1-g_2)^2}{V^++h}\right|\le \left|\frac{(g_1-g_2)^2}{V^+}\right|\le \left|\frac{(g_1-g_2)^2}{\mathcal{F}(g_1)-\mathcal{F}(g_2)}\right|<C,
\end{align} 
uniformly for $r\in[\varepsilon_0,1]$ and $h>0$. There exists a constant $\mathcal{C}\gg 1$ such that $|D^c|\ll 1$ satisfying
\begin{align}\label{yy17}
	\frac{j^2}{2\tau^2}\int_{D^c}r^2\left|\frac{(g_1-g_2)^2}{V^++h}\right|dr\le \frac{1}{2}\int_{\varepsilon_0}^{1}(g_1-g_2)\cdot \frac{V^+}{V^++h}dr.	
\end{align}
Combining \eqref{yy11}-\eqref{yy17}, we obtain
\begin{align}\label{yy18}
	&\frac{h}{2}\int_{\varepsilon_0}^{1}r^2\left|\left({\rm log}\left(1+\frac{V^+}{h}\right)\right)_r\right|^2dr+\frac{1}{2}\int_{\varepsilon_0}^{1}(g_1-g_2)\frac{V^+}{V^++h}dr\nonumber\\
	\le &\frac{j}{\tau}h\mathcal{C}\int_{\varepsilon_0}^{1}r^2\left|\left({\rm log}\left(1+\frac{V^+}{h}\right)\right)_r\right|dr+I_2.
\end{align}

Now, we give a detailed calculation of $I_2$, which is critical and more complex. Adding $\eqref{yy8}_1$ to $\eqref{yy8}_2$, it yields
\begin{align}\label{yy19}
	(\mathcal{F}(g))_r+\frac{j}{\tau}\left(\frac{1}{g}-\frac{1}{m}\right)-\frac{4}{r}=-(\mathcal{F}(m))_r.	
\end{align}
Integrating \eqref{yy19} on $(\varepsilon_0,r)$ and collecting $(g_1,m_1)$ and $(g_2,m_2)$, we derive
\begin{align}\label{yy20}
	&\mathcal{F}(m_1)-\mathcal{F}(m_2)\nonumber\\
	=&\mathcal{F}(g_2)-\mathcal{F}(g_1)-\frac{j}{\tau}\int_{\varepsilon_0}^{r}\left(\frac{1}{g_1(s)}-\frac{1}{g_2(s)}\right)ds+\frac{j}{\tau}\int_{\varepsilon_0}^{r}\left(\frac{1}{m_1(s)}-\frac{1}{m_2(s)}\right)ds,
\end{align}
and thus
\begin{align}\label{yy21}
	|\mathcal{F}(m_1)-\mathcal{F}(m_2)|\le |\mathcal{F}(g_1)-\mathcal{F}(g_2)|+\frac{j}{\tau}||g_1-g_2||_{L^1}+\frac{j}{\tau}||m_1-m_2||_{L^1}.
\end{align}
Integrating \eqref{yy21} on $[\varepsilon_0,1]$ yields
\begin{align}\label{yy23}
	||\mathcal{F}(m_1)-\mathcal{F}(m_2)||_{L^1}-\frac{j}{\tau}||m_1-m_2||_{L^1}\le||\mathcal{F}(g_1)-\mathcal{F}(g_2)||_{L^1}+\frac{j}{\tau}||g_1-g_2||_{L^1}.	
\end{align}
Since $m_1,m_2\in [\underline{m}+3,\bar{m}]$, there exists a constant $C_1\ge 1$ such that
\begin{align}\label{yy24}
	\frac{1}{C_1}||\mathcal{F}(m_1)-\mathcal{F}(m_2)||_{L^1}\le ||m_1-m_2||_{L^1}\le C_1||\mathcal{F}(m_1)-\mathcal{F}(m_2)||_{L^1}.
\end{align}
When $\frac{j}{\tau}\le\frac{1}{2C_1}$, we have
\begin{align}\label{yy25}
	||\mathcal{F}(m_1)-\mathcal{F}(m_2)||_{L^1}\le  2||\mathcal{F}(g_1)-\mathcal{F}(g_2)||_{L^1}+\frac{2j}{\tau}||g_1-g_2||_{L^1}.
\end{align}

Noting that
\begin{align}\label{yy26}
	I_2=\int_{\varepsilon_0}^{1}\frac{m_1-m_2}{\mathcal{F}(m_1)-\mathcal{F}(m_2)}\cdot(\mathcal{F}(m_1)-\mathcal{F}(m_2))\cdot\frac{V^+}{V^++h}dr,	
\end{align}
because if $\mathcal{F}(m_1)-\mathcal{F}(m_2)=0$, then $m_1-m_2=0$ due to the monotonicity of $\mathcal{F}$. There exists a constant $C_2\ge 1$ such that $\frac{1}{C_2}\le\frac{m_1-m_2}{\mathcal{F}(m_1)-\mathcal{F}(m_2)}\le C_2$. Plugging \eqref{yy20} to \eqref{yy26} and from \eqref{yy24}-\eqref{yy25}, it holds
\begin{align}\label{yy27}
	I_2&\le\int_{\varepsilon_0}^{1}\frac{m_1-m_2}{\mathcal{F}(m_1)-\mathcal{F}(m_2)}\cdot(\mathcal{F}(g_2)-\mathcal{F}(g_1))\cdot\frac{V^+}{V^++h}dr\nonumber\\
	&\,\,\,\,\,\,\,\,\,\,\,\,\,\,+\int_{\varepsilon_0}^{1}\frac{m_1-m_2}{\mathcal{F}(m_1)-\mathcal{F}(m_2)}\cdot|\psi(\tau,r)|\cdot\left|\frac{V^+}{V^++h}\right|dr\nonumber\\
	&\le \int_{\varepsilon_0}^{1}\frac{m_1-m_2}{\mathcal{F}(m_1)-\mathcal{F}(m_2)}\cdot(\mathcal{F}(g_2)-\mathcal{F}(g_1))\cdot\frac{V^+}{V^++h}dr\nonumber\\
	&\,\,\,\,\,\,\,\,\,\,\,\,\,\,+\left(\frac{2j^2}{\tau^2}C_1C_2+\frac{j}{\tau} C_2\right)\cdot||g_1-g_2||_{L^1}+\frac{2j}{\tau} C_1C_2\cdot||\mathcal{F}(g_1)-\mathcal{F}(g_2)||_{L^1},
\end{align}
where $\psi(\tau,r):=-\frac{j}{\tau}\int_{\varepsilon_0}^{r}\left(\frac{1}{g_1(s)}-\frac{1}{g_2(s)}\right)ds+\frac{j}{\tau}\int_{\varepsilon_0}^{r}\left(\frac{1}{m_1(s)}-\frac{1}{m_2(s)}\right)ds$ and $\frac{j}{\tau}\le \frac{1}{2C_1}$. Substituting \eqref{yy27} into \eqref{yy18}, we derive
\begin{align}\label{yy28}
	&\frac{h}{2}\int_{\varepsilon_0}^{1}r^2\left|\left({\rm log}\left(1+\frac{V^+}{h}\right)\right)_r\right|^2dr+\frac{1}{2}\int_{\varepsilon_0}^{1}(g_1-g_2)\cdot\frac{V^+}{V^++h}dr\nonumber\\
	&\,\,\,\,+\frac{1}{C_2}\int_{\varepsilon_0}^{1}(\mathcal{F}(g_1)-\mathcal{F}(g_2))\cdot\frac{V^+}{V^++h}dr\nonumber\\
	\le&\frac{j}{\tau}h\mathcal{C}\int_{\varepsilon_0}^{1}r^2\left|\left({\rm log}\left(1+\frac{V^+}{h}\right)\right)_r\right|dr+\left(\frac{2j^2}{\tau^2}C_1C_2+\frac{j}{\tau} C_2\right)\cdot||g_1-g_2||_{L^1}\nonumber\\
	&\,\,\,\,+\frac{2j}{\tau} C_1C_2\cdot||\mathcal{F}(g_1)-\mathcal{F}(g_2)||_{L^1}.
\end{align}

Similarly, taking $\varphi_h:=\frac{V^-}{V^-+h}$ as the test function for $h>0$ is small enough, and utilizing the same process as above, we get
\begin{align}\label{yy29}
	&\frac{h}{2}\int_{\varepsilon_0}^{1}r^2\left|\left({\rm log}\left(1+\frac{V^-}{h}\right)\right)_r\right|^2dr+\frac{1}{2}\int_{\varepsilon_0}^{1}(g_2-g_1)\cdot \frac{V^-}{V^-+h}dr\nonumber\\
	&\,\,\,\,+\frac{1}{C_2}\int_{\varepsilon_0}^{1}(\mathcal{F}(g_2)-\mathcal{F}(g_1))\cdot\frac{V^-}{V^-+h}dr\nonumber\\
	\le&\frac{j}{\tau} h\tilde{\mathcal{C}}\int_{\varepsilon_0}^{1}r^2\left|\left({\rm log}\left(1+\frac{V^-}{h}\right)\right)_r\right|dr+\left(\frac{2j^2}{\tau^2}C_1C_2+\frac{j}{\tau} C_2\right)\cdot||g_1-g_2||_{L^1}\nonumber\\
	&\,\,\,\,+\frac{2j}{\tau} C_1C_2\cdot||\mathcal{F}(g_1)-\mathcal{F}(g_2)||_{L^1},
\end{align}
where $\tilde{\mathcal{C}}>0$ is another constant. Combining \eqref{yy28} and \eqref{yy29}, we have
\begin{align}\label{yy30}
	&\frac{h}{2}\int_{\varepsilon_0}^{1}r^2\left|\left({\rm log}\left(1+\frac{V^+}{h}\right)\right)_r\right|^2dr+\frac{h}{2}\int_{\varepsilon_0}^{1}r^2\left|\left({\rm log}\left(1+\frac{V^-}{h}\right)\right)_r\right|^2dr\nonumber\\
	&\,\,\,\,+\frac{1}{2}\int_{\varepsilon_0}^{1}(g_1-g_2)\cdot \frac{V^+}{V^++h}dr+\frac{1}{2}\int_{\varepsilon_0}^{1}(g_2-g_1)\cdot \frac{V^-}{V^-+h}dr\nonumber\\
	&\,\,\,\,+\frac{1}{C_2}\int_{\varepsilon_0}^{1}(\mathcal{F}(g_1)-\mathcal{F}(g_2))\cdot\frac{V^+}{V^++h}dr+\frac{1}{C_2}\int_{\varepsilon_0}^{1}(\mathcal{F}(g_2)-\mathcal{F}(g_1))\cdot\frac{V^-}{V^-+h}dr\nonumber\\
	\le&\frac{j}{\tau} h\mathcal{C}\int_{\varepsilon_0}^{1}r^2\left|\left({\rm log}\left(1+\frac{V^+}{h}\right)\right)_r\right|dr+\frac{j}{\tau} h\tilde{\mathcal{C}}\int_{\varepsilon_0}^{1}r^2\left|\left({\rm log}\left(1+\frac{V^-}{h}\right)\right)_r\right|dr\nonumber\\
	&\,\,\,\,+\left(\frac{4j^2}{\tau^2}C_1C_2+\frac{2j}{\tau} C_2\right)\cdot||g_1-g_2||_{L^1}+\frac{4j}{\tau} C_1C_2\cdot||\mathcal{F}(g_1)-\mathcal{F}(g_2)||_{L^1}.
\end{align}
When $\frac{j}{\tau}:={\rm min}\left\{\frac{-C_2+\sqrt{C_2(C_1+C_2)}}{4C_1C_2},\frac{1}{8C_1C_2^2}\right\}$, we have
\begin{align}\label{yy31}
	&\frac{h}{2}\int_{\varepsilon_0}^{1}r^2\left|\left({\rm log}\left(1+\frac{V^+}{h}\right)\right)_r\right|^2dr+\frac{h}{2}\int_{\varepsilon_0}^{1}r^2\left|\left({\rm log}\left(1+\frac{V^-}{h}\right)\right)_r\right|^2dr\nonumber\\
	&\,\,\,\,+\frac{1}{2}\int_{\varepsilon_0}^{1}(g_1-g_2)\cdot \frac{V^+}{V^++h}dr+\frac{1}{2}\int_{\varepsilon_0}^{1}(g_2-g_1)\cdot \frac{V^-}{V^-+h}dr\nonumber\\
	&\,\,\,\,+\frac{1}{C_2}\int_{\varepsilon_0}^{1}(\mathcal{F}(g_1)-\mathcal{F}(g_2))\cdot\frac{V^+}{V^++h}dr+\frac{1}{C_2}\int_{\varepsilon_0}^{1}(\mathcal{F}(g_2)-\mathcal{F}(g_1))\cdot\frac{V^-}{V^-+h}dr\nonumber\\
	\le&\frac{j}{\tau} h\mathcal{C}\int_{\varepsilon_0}^{1}r^2\left|\left({\rm log}\left(1+\frac{V^+}{h}\right)\right)_r\right|dr+\frac{j}{\tau} h\tilde{\mathcal{C}}\int_{\varepsilon_0}^{1}r^2\left|\left({\rm log}\left(1+\frac{V^-}{h}\right)\right)_r\right|dr\nonumber\\
	&\,\,\,\,+\frac{1}{4}||g_1-g_2||_{L^1}+\frac{1}{2C_2}||w(g)_1-w(g)_2||_{L^1}.	
\end{align}
So there exists $h_0$, when $h<h_0$, we derive
\begin{align}\label{yy32}
	\frac{1}{4}||g_1-g_2||_{L^1}\le \frac{1}{2}\left(\int_{\varepsilon_0}^{1}(g_1-g_2)\cdot \frac{V^+}{V^++h}dr+\int_{\varepsilon_0}^{1}(g_2-g_1)\cdot \frac{V^-}{V^-+h}dr\right)
\end{align}
and
\begin{align}\label{yy33}
	&\frac{1}{2C_2}||\mathcal{F}(g_1)-\mathcal{F}(g_2)||_{L^1}\nonumber\\
	\le& \frac{1}{C_2}\left(\int_{\varepsilon_0}^{1}(\mathcal{F}(g_1)-\mathcal{F}(g_2))\cdot\frac{V^+}{V^++h}dr+\int_{\varepsilon_0}^{1}(\mathcal{F}(g_2)-\mathcal{F}(g_1))\cdot\frac{V^-}{V^-+h}dr\right).
\end{align}

Collecting \eqref{yy31}-\eqref{yy33}, we obtain
\begin{align}\label{yy34}
	&\frac{\varepsilon_0^2h}{2}\int_{\varepsilon_0}^{1}\left|\left({\rm log}\left(1+\frac{V^+}{h}\right)\right)_r\right|^2dr+\frac{\varepsilon_0^2h}{2}\int_{\varepsilon_0}^{1}\left|\left({\rm log}\left(1+\frac{V^-}{h}\right)\right)_r\right|^2dr\nonumber\\
	\le&\frac{j}{\tau} h\mathcal{C}\int_{\varepsilon_0}^{1}\left|\left({\rm log}\left(1+\frac{V^+}{h}\right)\right)_r\right|dr+\frac{j}{\tau} h\tilde{\mathcal{C}}\int_{\varepsilon_0}^{1}\left|\left({\rm log}\left(1+\frac{V^-}{h}\right)\right)_r\right|dr.	
\end{align}
That is
\begin{align}\label{yy35}
	&\int_{\varepsilon_0}^{1}\left|\left({\rm log}\left(1+\frac{V^+}{h}\right)\right)_r\right|^2dr+\int_{\varepsilon_0}^{1}\left|\left({\rm log}\left(1+\frac{V^-}{h}\right)\right)_r\right|^2dr\nonumber\\
	\le&\frac{2j \mathcal{C}}{\varepsilon_0^2\tau}\int_{\varepsilon_0}^{1}\left|\left({\rm log}\left(1+\frac{V^+}{h}\right)\right)_r\right|dr+\frac{2j \tilde{\mathcal{C}}}{\varepsilon_0^2\tau}\int_{\varepsilon_0}^{1}\left|\left({\rm log}\left(1+\frac{V^-}{h}\right)\right)_r\right|dr.	
\end{align}
Let $C_3:=\frac{2j \mathcal{C}}{\varepsilon_0^2\tau}$ and $C_4:=\frac{2j \tilde{\mathcal{C}}}{\varepsilon_0^2\tau}$, we get
\begin{align}\label{yy36}
	&\int_{\varepsilon_0}^{1}\left|\left({\rm log}\left(1+\frac{V^+}{h}\right)\right)_r\right|^2dr+\int_{\varepsilon_0}^{1}\left|\left({\rm log}\left(1+\frac{V^-}{h}\right)\right)_r\right|^2dr\nonumber\\
	\le&C_3\int_{\varepsilon_0}^{1}\left|\left({\rm log}\left(1+\frac{V^+}{h}\right)\right)_r\right|dr+C_4\int_{\varepsilon_0}^{1}\left|\left({\rm log}\left(1+\frac{V^-}{h}\right)\right)_r\right|dr.	
\end{align}
From \eqref{yy36}, it can be seen that there must eixst a sequence $\left\{h_i\right\},i=1,2,...$ with $h_i\to 0$, and at least one of the two inequalities 
\begin{align*}
\int_{\varepsilon_0}^{1}\left|\left({\rm log}\left(1+\frac{V^+}{h_i}\right)\right)_r\right|^2dr\le C_3\int_{\varepsilon_0}^{1}\left|\left({\rm log}\left(1+\frac{V^+}{h_i}\right)\right)_r\right|dr
\end{align*}
and 
\begin{align*}
\int_{\varepsilon_0}^{1}\left|\left({\rm log}\left(1+\frac{V^-}{h_i}\right)\right)_r\right|^2dr\le C_4\int_{\varepsilon_0}^{1}\left|\left({\rm log}\left(1+\frac{V^-}{h_i}\right)\right)_r\right|dr
\end{align*}
holds true. Without loss of generality, we suppose
\begin{align}\label{yy37}
	\int_{\varepsilon_0}^{1}\left|\left({\rm log}\left(1+\frac{V^+}{h_i}\right)\right)_r\right|^2dr\le C_3\int_{\varepsilon_0}^{1}\left|\left({\rm log}\left(1+\frac{V^+}{h_i}\right)\right)_r\right|dr	
\end{align}
holds for $h_i\to 0$. Denote $h_i$ also by $h$, according to \eqref{yy37} and using the H\"older inequality, we have 
\begin{align}\label{yy38}
	\left\|\left({\rm log}\left(1+\frac{V^+}{h}\right)\right)_r\right\|^2_{L^2}\le &C_3\int_{\varepsilon_0}^{1}\left|\left({\rm log}\left(1+\frac{V^+}{h}\right)\right)_r\right|dr\nonumber\\
	\le &C_3\left\|\left({\rm log}\left(1+\frac{V^+}{h}\right)\right)_r\right\|_{L^2}.
\end{align}

Finally, we obtain
\begin{align}\label{yy39}
	\left\|\left({\rm log}\left(1+\frac{V^+}{h}\right)\right)_r\right\|_{L^2}\le C_3,
\end{align}
for $h\to 0$. By \eqref{yy39} and the Poincar\'e inequality, we know $\left\|{\rm log}\left(1+\frac{V^+}{h}\right)\right\|_{L^2}$ is uniformly bounded with $h\to 0$, which is impossible, except that $V^+\equiv 0$. This means $\mathcal{F}(g_1)\le \mathcal{F}(g_2)$, and thus $g_1\le g_2$. And from \eqref{yy36}, it is easy to obtain
\begin{align}\label{yy40}
	\int_{\varepsilon_0}^{1}\left|\left({\rm log}\left(1+\frac{V^-}{h}\right)\right)_r\right|^2dr\le C_4\int_{\varepsilon_0}^{1}r^2\left|\left({\rm log}\left(1+\frac{V^-}{h}\right)\right)_r\right|dr.	
\end{align}
Similarly, we will get $V^-\equiv 0$, that is, $\mathcal{F}(g_1)\ge \mathcal{F}(g_2)$, then $g_1\ge g_2$.

Hence, we have proved $g_1=g_2$ on $[\varepsilon_0,1]$. According to \eqref{yy25}, we have $m_1=m_2$ on $[\varepsilon_0,1]$. Therefore, we have proved the uniqueness of subsonic solution to the system \eqref{eq2} when $\frac{j}{\tau}\ll    1$.\hfill $\Box$   

%%%%%%%%%%%%%%%%%%%%%%%%%%%%%%%%%%%%%%%%%%%%%%%%%%%%%%%%%%%%%%%%%%%%%%%%%%%%%%%%%%%%%%%%%%%%%%%%%%

\section*{Acknowledgements}
The research of M.~Mei was supported by NSERC grant RGPIN 2022--03374, the research of K.~Zhang was supported by National Natural Science Foundation of China grant 12271087, and the research of G.~Zhang was supported by National Natural Science Foundation of China grant 11871012.

%%%%%%%%%%%%%%%%%%%%%%%%%%%%%%%%%%%%%%%%%%%%%%%%%%%%%%%%%%%%%%%%%%%%%%%%%%%%%%%%%%%%%%%%%%%%%%%%%%

\appendix

\end{document}